\documentclass[a4paper,11pt]{article}

\usepackage{amsmath}
\usepackage{amssymb}
\usepackage{amsthm}
\usepackage{graphicx}
\usepackage{enumerate}
\usepackage{cite}
\usepackage[usenames]{color}
\usepackage{framed}

\theoremstyle{theorem}
\newtheorem{lemma}{Lemma}
\newtheorem{theorem}{Theorem}
\newtheorem{proposition}{Proposition}

\theoremstyle{definition}

\newcommand{\FF}{\mathcal{F}}
\newcommand{\conv}{\mathop{\mathrm{conv}}}
\newcommand{\R}{\mathbb{R}}

\title{Extended Formulations for Sparsity Matroids}

\author{%
Satoru Iwata\thanks{%
Department of Mathematical Informatics,
University of Tokyo,
Japan.
E-mail: iwata@mist.i.u-tokyo.ac.jp
}
\and
Naoyuki Kamiyama\thanks{%
Institute of Mathematics for Industry, 
Kyushu University, 
Japan.
E-mail: kamiyama@imi.kyusyu-u.ac.jp
}
\and
Naoki Katoh\thanks{%
Department of Architecture and Architectural Engineering, 
Kyoto University, Japan.
E-mail: naoki@archi.kyoto-u.ac.jp
}
\and
Shuji Kijima\thanks{%
Department of Informatics, 
Kyushu University.
E-mail: kijima@inf.kyushu-u.ac.jp
}
\and
Yoshio Okamoto\thanks{%
Department of Communication Engineering and Informatics,
Graduate School of Informatics and Engineering,
The University of Electro-Communications.
E-mail: okamotoy@uec.ac.jp
}
}

\date{
}

\begin{document}

\maketitle

\begin{abstract}
We show the existence of a
polynomial-size extended formulation for the base
polytope of a $(k,\ell)$-sparsity matroid.
For an undirected graph $G=(V,E)$,
the size of the formulation is 
$O(|V||E|)$ when $k \geq \ell$ and
$O(|V|^2 |E|)$ when $k \leq \ell$.
To this end, we employ the technique developed by Faenza et al.\ recently
that uses a randomized communication protocol.
\end{abstract}

\section{Introduction}

Let $k,\ell$ be integers such that $0 \leq \ell \leq 2k - 1$.
A simple undirected graph $G=(V,E)$ is \emph{$(k,\ell)$-sparse} if
$|F| \leq \max\{k|V(F)|-\ell, 0\}$ for every $F\subseteq E$, where
$V(F)$ refers to the set of vertices in $G$ that are incident to 
at least one of the edges in $F$.
Furthermore, $G$ is \emph{$(k,\ell)$-tight} if it is $(k,\ell)$-sparse and
$|E|= \max\{k|V|-\ell,0\}$.

Fix such $k$ and $\ell$.
For a simple undirected graph $G=(V,E)$, consider the family $\FF_{k,\ell}(G) \subseteq 2^E$
of all edge subsets $F\subseteq E$ such that $G_F=(V,F)$ is $(k,\ell)$-tight.
It is known that $\FF_{k,\ell}(G)$ is the base family of a matroid, called the \emph{$(k,\ell)$-sparsity matroid
of $G$} \cite{MR2785822}.\footnote{%
In some papers and books, a sparsity matroid is also called a \emph{count matroid} \cite{MR2848535}.
}
Sparsity matroids naturally appear at various places in discrete mathematics, especially
in combinatorial optimization and combinatorial rigidity:
The family of spanning trees forms a $(1,1)$-sparsity matroid;
The family of spanning $1$-trees forms a $(1,0)$-sparsity matroid;
The family of disjoint unions of $k$ spanning trees forms a $(k,k)$-sparsity matroid
\cite{Nash-Williams01011961};
The family of generically minimal rigid subgraphs in the plane forms a $(2,3)$-sparsity matroid
\cite{laman70};
A $(\binom{d+1}{2},\binom{d+1}{2})$-sparsity matroid is a key for analyzing the rigidity of $d$-dimensional body-and-bar
frameworks
\cite{Tay198495}.

We consider the following base polytope of a $(k,\ell)$-sparsity matroid:
\[
P_{k,\ell}(G) = \conv(\{ \chi_F \in \R^E \mid F \in \FF_{k,\ell}(G)\} ),
\]
where $\chi_F$ is the incidence vector of $F$.
Then,
$P_{k,\ell}(G)$ can be written as
\[
P_{k,\ell}(G) = 
\left\{
x \in \R^E
\,\middle|\,
\begin{array}{l}
\displaystyle \sum_{e \in E(X)}x_e \leq \max\{k|X|-\ell, 0\} \quad \text{for all } X \subseteq V,\\
\displaystyle \sum_{e \in E}x_e = \max\{k|V| - \ell, 0\},\\
x_e \geq 0 \quad \text{for all } e \in E
\end{array}
\right\},
\]
where $E(X)$ refers to the set of edges in $G$ that join two vertices in $X$ \cite{MR2848535}.
Note that if $X \subseteq V$ satisfies $k|X|-\ell \leq 0$, then
the inequality $\sum_{e\in E(X)}x_e \leq \max\{k|X|-\ell,0\}$ does not
define a facet of $P_{k,\ell}(G)$.

The main purpose of this paper is to give a compact extended formulation of $P_{k,\ell}(G)$.

Let $P \subseteq \R^d$ be a convex polytope.
An \emph{extended formulation} of $P$ is a polytope $Q \subseteq \R^k$, $k\geq d$, such 
that there exists a linear projection $\pi\colon \R^k \to \R^d$ with $\pi(Q)=P$.
The \emph{size} of an extended formulation $Q$ is the number of facets of $Q$.
For several cases, the size of an extended formulation can be much smaller than
the size of the original polytope.
For example, while the description of the spanning tree polytope by Edmonds \cite{edmondsgreedy}
uses exponentially many inequalities, Martin \cite{Martin1991119}
gave an extended formulation of that polytope of size $O(n^3)$, where
$n$ is the number of vertices.
However, as Rothvo\ss \cite{DBLP:journals/mp/Rothvoss13} showed, there exists
a 0/1-polytope $P$ in $\R^d$ such that any extended formulation of $P$ has size
exponential in $d$.
Rothvo\ss \cite{DBLP:journals/mp/Rothvoss13} further noted that
there exists a matroid polytope with such a property.
However, since his proof was based on a counting argument, no explicit matroid with such 
a property was given.
Since the spanning tree polytope is the base polytope of a matroid, 
these previous results give rise to a question on which classes of matroids admit compact
extended formulations and which classes do not.
Since $P_{k,\ell}(G)$ is the base polytope of a matroid, this paper gives another
class of matroids that have compact extended formulations.
We here note that Goemans \cite{goemanspermutahedron} 
gave a compact extended formulation of permutahedra, which are the
base polytopes of polymatroids.

Our proof is based on a technique 
proposed by Faenza et al.\ \cite{FaenzaFGT12}.
In their method, we design a 
randomized communication protocol that computes the slack matrix
of a polytope in expectation.
The number of exchanged bits will determine the size of an extended formulation.
Indeed, Faenza et al.\ \cite{FaenzaFGT12}
gave an alternative proof of a result by Martin \cite{Martin1991119} on the spanning tree
polytope, and our proof can be seen as a generalization of their proof.
Our protocol gives an extended formulation of size $O(|V||E|)$ when
$k \geq \ell$, and $O(|V|^2 |E|)$ when $k \leq \ell$.

\section{A randomized communication protocol}

Here, we give a brief overview of the technique
based on a randomized communication protocol by Faenza et al.\ \cite{FaenzaFGT12}.
For a detailed and precise account, we refer to the original paper.

Let $P \subseteq \R^d$ be a polytope, with its extreme points
$\{z_1,\dots,z_n\} \subseteq \R^d$.
Suppose that $P$ is represented as
\[
P = \{ x \in \R^d \mid Ax \leq b\}
\]
for some $A \in \R^{m\times d}$ and $b \in \R^m$.
We denote the $i$-th row of $A$ by $a_i$ and the $i$-th entry of $b$ by $b_i$.
We assume that each inequality $a_i x \leq b_i$ in the description defines a facet of $P$.
With this assumption, the description of $P$ is unique up to scaling (i.e., a scalar multiplication
of each row).

The \emph{slack matrix} of $P$ is a non-negative matrix $S \in \R^{m\times n}$ defined as
\[
S_{i,j} = b_i - a_i z_j
\]
for every $i \in\{1,\dots,m\}$ and $j \in \{1,\dots,n\}$.
Namely, each row of $S$ corresponds to an inequality (or a facet) that defines $P$, and
each column of $S$ corresponds to an extreme point of $P$.
Since an extreme point belongs to $P$, it holds $Az_j \leq b$ for each $j \in \{1,\dots,n\}$,
and therefore $S_{i,j} = b_i - a_i z_j \geq 0$ for every 
$i \in\{1,\dots,m\}$ and $j \in \{1,\dots,n\}$.

We consider the following (cooperative) game.
Alice and Bob act as players.
Before the game starts, 
the polytope $P$ is known to both players with its set of extreme points and its
set of facet-defining inequalities.
When the game starts, 
Alice secretly receives one of the facet-defining inequalities, for example, the $i$-th inequality
$a_i x \leq b_i$, which is not visible to Bob, and 
Bob secretly receives one of the extreme points, for example, the $j$-th extreme point $z_j$,
which is not visible to Alice.
In the course of the game, they are allowed to communicate by exchanging information.
They are also allowed to use private random bits.
The players can agree with the way of communication (or a randomized communication protocol) 
before they start a game.
The goal of Alice and Bob is to output the $i,j$-th entry $S_{i,j}$ of the slack matrix at the
end of the protocol.
Note that the output can be made by either player, but at the end of the protocol
when the player outputs a value, he or she cannot use a random bit.

Since they are allowed to use random bits, the output may be wrong.
Therefore, we only require them to output $S_{i,j}$ in expectation.
Namely, the protocol \emph{computes $S$ in expectation} if
the expectation of the output (over the random bits used by
Alice and Bob) is equal to $S_{i,j}$ when Alice owns
$a_i x \leq b_i$ and Bob owns $z_j$.

In a trivial protocol, Bob sends the whole information of $z_j$
to Alice, and Alice directly computes $b_i - a_i z_j$ and output it.
This computes $S$ without any error, but the number of exchanged
bits is large.
We want to design a protocol that computes $S$ in expectation with
small number of exchanged bits.

Faenza et al.\ \cite{FaenzaFGT12}
characterized the size of an extended formulation
by the complexity of such a communication protocol.

\begin{proposition}[Faenza, Fiorini, Grappe, Tiwary \cite{FaenzaFGT12}]
\label{ffgt}
Let $P$ be a polytope with the slack matrix $S$.
Then, 
there exists a randomized communication protocol that computes $S$ in expectation with exchanging 
at most $\lceil \log_2 r\rceil$ bits of information
if and only if there exists an extended formulation of $P$
of size at most $r$.
\end{proposition}

Their proof is indeed constructive, and from a randomized
communication protocol we may obtain an extended formulation.

\section{A protocol for sparsity matroids: when $k \geq \ell$}

To use Proposition \ref{ffgt},  we design a randomized communication protocol
that computes the slack matrix of $P_{k,\ell}(G)$ in expectation.
We may assume that $|V|\geq 2$ since otherwise the size of an extended formulation
is constant.

Alice holds a vertex subset $X \subseteq V$ with $k|X|-\ell \geq 0$ and $|X| \geq 2$
that specifies a facet of $P_{k,\ell}(G)$,\footnote{%
We do not require that \emph{all} such vertex subsets $X$ define facets of $P_{k,\ell}(G)$, but we \emph{assume} the set $X$ that Alice holds defines a facet.
} and
Bob holds an edge subset $F \in \FF_{k,\ell}(G)$ that specifies an extreme point of $P_{k,\ell}(G)$.
Then, the slack $S_{X,F}$ is evaluated as 
\[
S_{X,F} = k|X| - \ell - |F \cap E(X)|.
\]

We first describe a protocol when $k \geq \ell$.
In the sequel, $\rho(v)$ denotes the in-degree of $v$.

\begin{framed}
\noindent
{\bf Protocol A}
\begin{description}
\item[Step 1:] Alice sends a vertex $x \in X$ to Bob.
\item[Step 2:] 
Bob orients the edges of $F$ in such a way that $\rho(x)=k-\ell$ and $\rho(z)=k$
for all $z \in V\setminus\{x\}$.
Then, Bob picks an oriented edge $(u,v)$ of $F$ uniformly at random and sends it
to Alice.
\item[Step 3:] Alice outputs $k|V|-\ell$ if $u\not\in X$ and $v\in X$, and
outputs $0$ otherwise.
\end{description}
\end{framed}

Note that $k-\ell \geq 0$ since $k \geq \ell$.

The number of exchanged bits is $\lceil \log|V| \rceil$ at Step 1
and $\lceil \log |E| \rceil + 1$ at Step 2.
Thus, by Proposition \ref{ffgt},
this protocol gives an extended formulation of $P_{k,\ell}(G)$ of size
$O(|V| |E|)$ if this is a correct protocol.

To complete the proof,
We will argue (1) 
the protocol can always be performed (in particular, 
the orientation by Bob can always be found), and
(2) the protocol computes the slack matrix
in expectation.

\subsection{Feasibility of Step 2}

We will show that Bob can always orient the edges of $F$ as described in Step~2.
\begin{lemma}
\label{lem:orientation2}
Let $k,\ell$ be integers such that $0 \leq \ell \leq k$, 
$H=(V,F)$ be a $(k,\ell)$-tight graph with $|V|\geq 2$,
and $x \in V$ be an arbitrary vertex
of $H$.
Then, $H$ has an orientation $\vec{H}=(V,\vec{F})$ such that
$\rho(x)=k-\ell$ and $\rho(z) = k$ for all $z \in V\setminus \{x\}$.
\end{lemma}

To prove the lemma, we use the following theorem by Hakimi \cite{hakimi}.
\begin{lemma}[Hakimi]
\label{hakimi}
Let $H=(V,F)$ be an undirected graph and
$m(v)$ be a non-negative integer for every $v \in V$.
Then, $H$ has an orientation such that $\rho(v) = m(v)$ for all $v$ if and only if
\begin{equation}
|F|=\sum_{v \in V}m(v) \qquad \text{ and } \qquad
|F(X)| \leq \sum_{v \in X} m(v) \quad \text{ for all } X \subseteq V,
\label{eq:hakimi}
\end{equation}
where $F(X)$ denotes the set of edges in $F$ that joins two vertices of $X$.
\end{lemma}

\begin{proof}[Proof of Lemma \ref{lem:orientation2}]
It suffices to verify the condition (\ref{eq:hakimi}).
Let $x\in V$ be an arbitrary vertex, and
set
$m(x)=k-\ell$, and $m(z) = k$ for all $z \in V\setminus \{x\}$.
Then, the first condition follows since
\[
\sum_{v \in V}m(v) = k(|V|-1)+(k-\ell) = k|V| - \ell = |F|,
\]
where 
the last equality is a consequence of $(k,\ell)$-tightness.
To see the second condition, let $X \subseteq V$ be an arbitrary vertex subset.
If $|X| \geq 1$, then we have
\[
\sum_{v \in X}m(v) \geq k(|X|-1)+(k-\ell) = k|X| - \ell \geq |F(X)|.
\]
Here, the first equality follows since $k \geq \ell$ and
the last inequality is a consequence of $(k,\ell)$-sparsity.
If $|X| = 0$, then
\[
\sum_{v \in X}m(v) = 0 = |F(X)|.
\]
Thus, the condition (\ref{eq:hakimi}) holds in both cases.
\end{proof}

\subsection{Correctness}

We now prove that the expected output of the protocol is 
the slack $S_{X,F}=k|X| - \ell - |F \cap E(X)|$.
Let $\vec{F}$ be the orientation of $F$ obtained at Step 2.
Then, since $x \in X$, it holds that
\[
\sum_{v \in X} \rho(v) = k(|X|-1) + (k-\ell) = k|X| -\ell. 
\]
On the other hand, since the left-hand side counts the number of oriented edge
that has its head in $X$,
\begin{align*}
\sum_{v \in X} \rho(v) 
 &= |\{(u,v) \in \vec{F} \mid u \in X, v \in X\}| + |\{(u,v) \in \vec{F} \mid u \not\in X, v \in X\}|\\
 &= |F\cap E(X)| + |\{(u,v) \in \vec{F} \mid u \not\in X, v \in X\}|.
\end{align*}
Therefore, 
the slack is the number of edges $(u,v)$ in $\vec{F}$ 
that enter $X$ from $V\setminus X$:
\[
S_{X,F}=k|X|-\ell-|F\cap E(X)| = |\{(u,v) \in \vec{F} \mid u\not\in X, v \in X\}|.
\]

At Step 2, an edge $(u,v) \in \vec{F}$ is chosen uniformly at random.
Therefore,
\begin{eqnarray*}
\Pr[u \not\in X \text{ and } v \in X \mid (u,v) \in \vec{F}]
&=&
\frac{|\{(u,v) \in \vec{F} \mid u\not\in X, v \in X\}|}{|F|}\\
&=&
\frac{S_{X,F}}{k|V|-\ell}.
\end{eqnarray*}
Therefore, the expected output is exactly
\[
(k|V|-\ell)\cdot \frac{S_{X,F}}{k|V|-\ell} + 0 \cdot \left( 1 - \frac{S_{X,F}}{k|V|-\ell}\right) = S_{X,F}.
\]

By the argument above, together with Proposition \ref{ffgt}, we have finished the proof of the
following theorem.

\begin{theorem}
Let $k,\ell$ be integers such that $0 \leq \ell \leq k$, and
$G=(V,E)$ be a simple undirected graph.
Then, the polytope $P_{k,\ell}(G)$ of the $(k,\ell)$-tight subgraphs
of $G$ has an extended formulation of size $O(|V||E|)$.
\qed
\end{theorem}

Faenza et al.\ \cite{FaenzaFGT12} gave a randomized communication protocol for
the slack matrix of a spanning tree polytope, which yields an extended formulation
of size $O(|V||E|)$.
They also used an orientation of the graph as in our argument, and in this
sense our protocol is an extension of theirs.
Remember that the family of spanning trees forms a $(1,1)$-sparsity matroid.

\section{A protocol for sparsity matroids: when $k \leq \ell$}

Protocol A does not work when $k < \ell$.
Therefore, we provide another protocol, as shown below, to
give an extended formulation of size
$O(|V|^2|E|)$.

\begin{framed}
\noindent
{\bf Protocol B}
\begin{description}
\item[Step 1:] Alice sends two arbitrary vertices $x,y \in X$ to Bob.
\item[Step 2:] Bob orients the edges of $F$ in such a way that $\rho(x)=0$,
$\rho(y) = 2k-\ell$ and $\rho(z) = k$ for all $z \in V\setminus \{x,y\}$.
Then, Bob picks an oriented edge $(u,v)$ of $F$ uniformly at random and sends it
to Alice.
\item[Step 3:] Alice outputs $k|V|-\ell$ if $u\not\in X$ and $v\in X$, and
outputs $0$ otherwise.
\end{description}
\end{framed}

Note that $k \geq 2k-\ell$ since $k \leq \ell$.

The number of exchanged bits is $2 \lceil \log|V| \rceil$ at Step 1
and $\lceil \log |E| \rceil + 1$ at Step 2
(an extra one bit is needed for specifying the orientation).
Thus, by Proposition \ref{ffgt},
this protocol gives an extended formulation of $P_{k,\ell}(G)$ of size
$O(|V|^2 |E|)$ if this is a correct protocol.

To complete the proof, we will again argue (1) the protocol can always be performed (in particular, 
the orientation by Bob can always be found), and (2) the protocol computes the slack matrix
in expectation,
in the same way as in the case where $k \geq \ell$.
The proofs are almost verbatim, but
for the sake of conciseness, we will repeat the arguments below.

\subsection{Feasibility of Step 2}

We will show that Bob can always orient the edges of $F$ as described in Step~2.
\begin{lemma}
\label{lem:orientation}
Let $k,\ell$ be integers such that $0 \leq k \leq \ell \leq 2k - 1$,
$H=(V,F)$ be a $(k,\ell)$-tight graph with $|V|\geq　2$,
and $x,y \in V$ two arbitrary vertices
of $H$.
Then, $H$ has an orientation $\vec{H}=(V,\vec{F})$ such that
$\rho(x)=0$, $\rho(y)=2k-\ell$, and $\rho(z) = k$ for all $z \in V\setminus \{x,y\}$.
\end{lemma}

\begin{proof}
We again use Hakimi's result (Lemma \ref{hakimi}).
Then, 
it suffices to verify the condition (\ref{eq:hakimi}).
Let $x,y\in V$ be arbitrary vertices, and
set
$m(x)=0$, $m(y)=2k-\ell$, and $m(z) = k$ for all $z \in V\setminus \{x,y\}$.
Then, the first condition follows since
\[
\sum_{v \in V}m(v) = k(|V|-2)+(2k-\ell) = k|V| - \ell = |F|,
\]
where 
the last equality is a consequence of $(k,\ell)$-tightness.
To see the second condition, let $X \subseteq V$ be an arbitrary vertex subset.
If $|X| \geq 2$, then we have
\[
\sum_{v \in X}m(v) \geq k(|X|-2)+(2k-\ell) = k|X| - \ell \geq |F(X)|.
\]
Here, 
the first equality follows since $k \leq \ell$ and 
the last inequality is a consequence of $(k,\ell)$-sparsity.
If $|X| \leq 1$, then
\[
\sum_{v \in X}m(v) \geq 0 = |E(X)|.
\]
Thus, the condition (\ref{eq:hakimi}) holds in both cases.
\end{proof}

\subsection{Correctness}

We now prove that the expected output of the protocol is 
the slack $S_{X,F}=k|X| - \ell - |F \cap E(X)|$.
Let $\vec{F}$ be the orientation of $F$ obtained at Step 2.
Then, since $x,y \in X$, it holds that
\[
\sum_{v \in X} \rho(v) = k(|X|-2) + (2k-\ell) = k|X| -\ell. 
\]
Since
\[
|F\cap E(X)| = |\{(u,v) \in \vec{F} \mid u \in X, v \in X\}|,
\]
the slack is the number of edges $(u,v)$ in $\vec{F}$ 
that enter $X$ from $V\setminus X$:
\[
S_{X,F}=k|X|-\ell-|F\cap E(X)| = |\{(u,v) \in \vec{F} \mid u\not\in X, v \in X\}|.
\]

At Step 2, an edge $(u,v) \in \vec{F}$ is chosen uniformly at random.
Therefore,
\begin{eqnarray*}
\Pr[u \not\in X \text{ and } v \in X \mid (u,v) \in \vec{F}]
&=&
\frac{|\{(u,v) \in \vec{F} \mid u\not\in X, v \in X\}|}{|F|}\\
&=&
\frac{S_{X,F}}{k|V|-\ell}.
\end{eqnarray*}
Therefore, the expected output is exactly
\[
(k|V|-\ell)\cdot \frac{S_{X,F}}{k|V|-\ell} + 0 \cdot \left( 1 - \frac{S_{X,F}}{k|V|-\ell}\right) = S_{X,F}.
\]

By the argument above, together with Proposition \ref{ffgt}, we have finished the proof of the
following theorem.
\begin{theorem}
Let $k,\ell$ be integers such that $0 \leq k \leq \ell \leq 2k - 1$, and
$G=(V,E)$ be a simple undirected graph.
Then, the polytope $P_{k,\ell}(G)$ of the $(k,\ell)$-tight subgraphs
of $G$ has an extended formulation of size $O(|V|^2|E|)$.
\qed
\end{theorem}

\section*{Acknowledgements}

The problem in this paper was partially discussed at 
ELC Workshop on Polyhedral Approaches: Extension Complexity and Pivoting Lower Bounds,
held in Kyoto, Japan (June 2013).
The authors thank the organizers and the participants of the workshop for 
stimulation to this work.
Special thanks go to Hans Raj Tiwary for sharing
his insights on randomized communication protocols.
This work is supported by
Grant-in-Aid for Scientific Research from Ministry of Education,
Science and Culture, Japan and Japan Society for the Promotion of Science,
and
the ELC project (Grant-in-Aid for
Scientific Research on Innovative Areas, MEXT Japan).

\bibliographystyle{siam}
\bibliography{xcsparse}

\begin{thebibliography}{10}

\bibitem{edmondsgreedy}
{\sc J.~Edmonds}, {\em Matroids and the greedy algorithm}, Mathematical
  Programming, 1 (1971), pp.~127--136.

\bibitem{FaenzaFGT12}
{\sc Y.~Faenza, S.~Fiorini, R.~Grappe, and H.~R. Tiwary}, {\em Extended
  formulations, nonnegative factorizations, and randomized communication
  protocols}, Mathematical Programming,  (2014).
\newblock to appear.

\bibitem{MR2848535}
{\sc A.~Frank}, {\em Connections in Combinatorial Optimization}, vol.~38 of
  Oxford Lecture Series in Mathematics and its Applications, Oxford University
  Press, Oxford, 2011.

\bibitem{goemanspermutahedron}
{\sc M.~X. Goemans}, {\em Smallest compact formulation for the permutahedron},
  Mathematical Programming,  (2014).
\newblock to appear.

\bibitem{hakimi}
{\sc S.~L. Hakimi}, {\em On the degrees of the vertices of a directed graphs},
  Journal of the Franklin Institute, 279 (1965), pp.~290--308.

\bibitem{laman70}
{\sc G.~Laman}, {\em On graphs and rigidity of plane skeletal structures},
  Journal of Engineering Mathematics, 4 (1970), pp.~331--340.

\bibitem{Martin1991119}
{\sc R.~K. Martin}, {\em Using separation algorithms to generate mixed integer
  model reformulations}, Operations Research Letters, 10 (1991), pp.~119--128.

\bibitem{Nash-Williams01011961}
{\sc C.~S. J.~A. Nash-Williams}, {\em Edge-disjoint spanning trees of finite
  graphs}, Journal of the London Mathematical Society, s1-36 (1961),
  pp.~445--450.

\bibitem{DBLP:journals/mp/Rothvoss13}
{\sc T.~Rothvo{\ss}}, {\em Some 0/1 polytopes need exponential size extended
  formulations}, Math. Program., 142 (2013), pp.~255--268.

\bibitem{MR2785822}
{\sc I.~Streinu and L.~Theran}, {\em Natural realizations of sparsity
  matroids}, Ars Math. Contemp., 4 (2011), pp.~141--151.

\bibitem{Tay198495}
{\sc T.-S. Tay}, {\em Rigidity of multi-graphs. {I}. {L}inking rigid bodies in
  $n$-space}, Journal of Combinatorial Theory, Series B, 36 (1984),
  pp.~95--112.

\end{thebibliography}

\end{document}